\documentclass[11pt,reqno]{amsart}
\usepackage{amssymb}
\usepackage{amsmath}
\usepackage{amsrefs}
\usepackage{amsthm}
\usepackage{array}
\usepackage{rotating}
\usepackage{verbatim}
\usepackage {color}
\usepackage{enumitem} 

\newlength{\tabwidth}
\newlength{\tabheight}
\newlength{\tabrule}
\newlength{\tabwidthx}
\newlength{\tabheightx}

\def\gentabbox#1#2#3#4{\vbox to \tabheight{\setlength{\tabrule}{#3}%
  \setlength{\tabwidthx}{#1\tabwidth}\addtolength{\tabwidthx}{\tabrule}%

\setlength{\tabheightx}{#2\tabheight}\addtolength{\tabheightx}{-\tabheight}%
  \hbox to #1\tabwidth{%
    \hspace{-0.5\tabrule}\rule{\tabrule}{#2\tabheight}\hspace{-\tabrule}%
    \vbox to #2\tabheight{\hsize=\tabwidthx%
      \vspace{-0.5\tabrule}\hrule width\tabwidthx height\tabrule%
      \vspace{-0.5\tabrule}\vfil%
      \hbox to \tabwidthx{\hss#4\hss}%
        \vfil\vspace{-0.5\tabrule}%
      \hrule width\tabwidthx height\tabrule\vspace{-0.5\tabrule}}%
    \hspace{-\tabrule}\rule{\tabrule}{#2\tabheight}\hspace{-0.5\tabrule}}%
  \vspace{-\tabheightx}}}
\def\genblankbox#1#2{\vbox to \tabheight{\vfil\hbox to
#1\tabwidth{\hfil}}}
\def\tabbox#1#2#3{\gentabbox{#1}{#2}{0.4pt}{\strut #3}}

\catcode`\:=13 \catcode`\.=13 \catcode`\;=13 
\catcode`\>=13 \catcode`\^=13
\def:#1\\{\hbox{$#1$}}
\def.#1{\tabbox{1}{1}{$#1$}}
\def>#1{\tabbox{2}{1}{$#1$}}
\def^#1{\tabbox{1}{2}{$#1$}}
\def;{\genblankbox{1}{1}\relax}
\catcode`\:=12 \catcode`\.=12 \catcode`\;=12 
\catcode`\>=12 \catcode`\^=7

\newenvironment{tableau}{\bgroup\catcode`\:=13 \catcode`\.=13
  \catcode`\;=13 \catcode`\>=13 \catcode`\^=13
  \setlength{\tabheight}{3ex}\setlength{\tabwidth}{3ex}%
  \def\b##1##2##3{\gentabbox{##1}{##2}{1.2pt}{\vbox{##3}}}%
  \def\n##1##2##3{\gentabbox{##1}{##2}{0.4pt}{\vbox{##3}}}%
  \vbox\bgroup\offinterlineskip}{\egroup\egroup}

\newcommand{\GL}{\mathrm{GL}}
\newcommand{\U}{\mathrm{U}}

\newcommand{\ol}{\overline}

\newcommand{\sm}{\mathrm{sm}}

\newcommand{\IC}{\mathrm{IC}}

\newcommand{\SL}{\mathrm{SL}}
\newcommand{\Sp}{\mathrm{Sp}}
\newcommand{\CV}{\mathrm{CV}}
\newcommand{\CC}{\mathrm{CC}}

\newcommand{\Ann}{\mathrm{Ann}}

\newcommand{\ks}{\mathrm{KS}}

\newcommand{\lra}{\longrightarrow}

\newcommand{\bs}{\backslash}

\newcommand{\opp}{-}

\newtheorem{prop}{Proposition}[section]

\newtheorem{lemma}[prop]{Lemma}
\newtheorem{theorem}[prop]{Theorem}

\numberwithin{equation}{section}

\newtheorem{corollary}[prop]{Corollary}

\theoremstyle{definition}

\newtheorem{remark}[prop]{Remark}

\newtheorem{definition}[prop]{Definition}

\newcommand{\frb}{\mathfrak{b}}

\newcommand{\frg}{\mathfrak{g}}
\newcommand{\frh}{\mathfrak{h}}

\newcommand{\frk}{\mathfrak{k}}
\newcommand{\frl}{\mathfrak{l}}

\newcommand{\frn}{\mathfrak{n}}

\newcommand{\frp}{\mathfrak{p}}
\newcommand{\frq}{\mathfrak{r}} 

\newcommand{\frs}{\mathfrak{s}}

\newcommand{\fru}{\mathfrak{u}}

\newcommand{\bbC}{\mathbb{C}}

\newcommand{\bbN}{\mathbb{N}}

\newcommand{\bbR}{\mathbb{R}}

\newcommand{\bbZ}{\mathbb{Z}}

\newcommand{\caB}{\mathcal{B}}

\newcommand{\caD}{\mathcal{D}}

\newcommand{\caL}{\mathcal{L}}

\newcommand{\caN}{\mathcal{N}}

\newcommand{\caR}{\mathcal{R}}

\newcommand{\caZ}{\mathcal{Z}}

\newcommand{\beq}{\begin{equation}}
\newcommand{\eeq}{\end{equation}}

\def\ad {\mathop{\hbox {ad}}\nolimits}

\def \det {\mathrm{det}}

\def\Id {\mathop{\hbox{Id}}\nolimits}

\begin{document}
\title[reducible characteristic cycles]{reducible characteristic cycles of harish-chandra modules for $\U(p,q)$ and the kashiwara-saito singularity}
\author{Leticia Barchini}
\address{Department of Mathematics, Oklahoma State University, Stillwater, OK 74078}
\email{leticia@math.okstate.edu}

\author{Petr Somberg}
\address{Mathematical Institute, Charles University, Sokolovask\'{a} 83, 18000 Prague 8}
\email{somberg@karlin.mff.cuni.cz}

\author{Peter E.~Trapa}
\address{Department of Mathematics, University of Utah, Salt Lake City, UT 84112-0090}
\email{ptrapa@math.utah.edu}

\thanks{L.~Barchini and P.~Somberg acknowledge the financial support from the grant GACR P201/12/G028.  P.~Trapa was supported by
the NSF grant DMS-1302237 and the Simons Foundation.}

\begin{abstract}
We give examples of reducible characteristic cycles for irreducible Harish-Chandra modules for $U(p,q)$ by 
analyzing a four-dimensional singular subvariety of $\bbC^8$.  We relate this singularity to the Kashiwara-Saito
singularity arising for Schubert varieties for $\GL(8,\bbC)$ with reducible characteristic cycles, as well as recent 
related examples of Williamson.  In particular, this gives another explanation of the simple highest weight modules
for $\frg\frl(12,\bbC)$ with reducible associated varieties that Williamson discovered.
\end{abstract}

\maketitle

\section{introduction}
Characteristic cycles of Harish-Chandra modules are delicate and important geometric invariants of central interest.  
To take but one example, they are at the heart of the Adams-Barbasch-Vogan 
approach to the Arthur conjectures for real groups \cite{abv}.  They are notoriously difficult to compute in any generality.

In the setting of highest weight modules and Schubert varieties, Kazhdan-Lusztig \cite[\S 7]{kl} first raised the question of whether characteristic cycles in Type A were always
irreducible.  (Borho-Brylinski's results \cite{bb} implied that this is equivalent to the question of whether characteristic cycles of Harish-Chandra bimodules for $\GL(n,\bbC)$ were always irreducible.)  This remained open until Kashiwara-Saito \cite{ks} found a surprising example in $G= \GL(8,\bbC)$.  The relevant singularity controlling this example was an an eight-dimensional subvariety of $\bbC^{16}$ (recalled in \eqref{e:ks} below) and was called the Kashiwara-Saito singularity in \cite[Example 1.14]{w:icm}.  The reducibility example of Kashiwara-Saito had the property that the two irreducible components
had different moment map images in the dual $\frg^*$ of the Lie algebra of $G$.  For this reason, results of Joseph \cite{joseph-polys} show that the reducibilty could not be detected by
considering associated varieties in $\frg^*$ of simple highest weight modules for $\frg$ (in the sense of \cite{joseph-av}).

The next development in this direction was the example of Williamson \cite{williamson} for $\GL(12,\bbC)$ of a reducible characteristic cycle, this time with both components having the same moment map image.  Representation theoretically, the results of Joseph \cite{joseph-polys} now implied that this example corresponded to a simple highest weight module for $\frg\frl(12,\bbC)$ whose associated variety was reducible,
something which was thought not to be possible for $\frg\frl(n)$.  Again it was the Kashiwara-Saito singularity that controlled Williamson's example.  

The purpose of this paper is to give the first examples of reducible characteristic cycles for irreducible Harish-Chandra modules for $U(p,q)$ with trivial
infinitesimal character.  These examples are defined in terms of $\GL(p,\bbC) \times \GL(q,\bbC)$ orbits on the flag variety for $\GL(n,\bbC)$ with $p+q=n$.
For $(p,q)=(4,4)$ we find an orbit whose closure has reducible characteristic cycle (Theorem \ref{t:main}(a)); like the Kashiwara-Saito example, the two components
have different moment map images.  We then find a closely related example for $(p,q)=(6,6)$ where, like the Williamson example,  both components have the same moment map image (Theorem \ref{t:main}(b)).  Representation theoretically, the Harish-Chandra
module arising in the $U(6,6)$ example are cohomologically induced in the good range from the $U(4,4)$ example.

In light of the examples of Kashiwara-Saito and Williamson, the existence of our examples is perhaps not surprising.   One might reasonably expect that the Kashiwara-Saito singularity
is showing up in both places.   Here there is a nice surprise:  the relevant singularity controlling our reducibility is in fact a four-dimensional subvariety of $\bbC^8$ (given in Lemma \ref{l:maina-2} and  Lemma \ref{new}).  The close relationship is explained in Sections \ref{s:ks} and 
 \ref{s:williamson}.  In particular, the reducibility in the Kashiwara-Saito and Williamson examples could be deduced from the lower-dimensional considerations here (though, to be fair, we never would
 have found our examples without \cite{ks} and \cite{williamson} which guided how we looked for them).

Here is an outline of the paper.   In Section \ref{s:slices}, we develop the general framework for normal slices to $K$ orbits on the flag variety for $G$ where $K$ is the fixed
points of an involution of a connected reductive algebraic group $G$.  For orientation, we begin with the perhaps more familiar setting of Bruhat cells in Section \ref{s:bruhat} before turning to $K$ orbits in Section \ref{s:korbits}.  (The former can be recovered as a special case of the latter by taking $K$ to be the diagonal copy of $G$ acting on pairs of flags.)  We also look at the setting for cohomological induction (Section \ref{s:coh-induction}) proving a useful result (Proposition \ref{p:coh-induction}) with  applications to representation theory (Corollary \ref{c:coh-induction-rep}) of independent interest. 

We  state our main results in Theorem \ref{t:main}.  We apply the theory of Section \ref{s:slices} to explicitly analyze the intersection of a particular orbit closure with a relevant eight-dimensional normal slice.  The singular intersection $Y$ is described in Lemma \ref{l:maina-2} based on general results of Wyser-Yong \cite{wyser-yong}.  A more convenient realization is given as the variety $Z$ in Lemma \ref{new}.  We then use a criterion of Braden \cite{braden} (Lemma \ref{l:braden}) to prove the reducibility of the characteristic cycle of $Y$ (Lemma \ref{l:maina-3})
from which Theorem \ref{t:main}(a) follows.
Proposition \ref{p:coh-induction} is then used to deduce Theorem \ref{t:main}(b) in Section \ref{s:proofb}.  Further explicit details are given in Section \ref{s:moment-map}.

In Section \ref{s:ks}, we relate our singularity $Z$ with the Kashiwara-Saito singularity $Z_\ks$.  We conclude in Section \ref{s:williamson} 
by explaining the relationship with the examples of Williamson by relating certain Goldie rank polynomials to certain Joseph polynomials.  In particular, this gives another explanation of the simple highest weight modules
for $\frg\frl(12,\bbC)$ with reducible associated varieties that Williamson discovered.

\section{normal slices}
\label{s:slices}
\subsection{Bruhat cells}
\label{s:bruhat}

Fix a complex reductive algebraic group $G$ with Lie algebra $\frg$.  Fix a Borel subalgebra $\frb = \frh \oplus \frn$ with corresponding positive root system $\Delta^+ \subset \Delta:=\Delta(\frh,\frg) \subset \frh^*$.  Write $B$ for the normalizer in $G$ of $\frb$, and similarly for $H$ and $N$.   Let $N^\opp$ denote the opposite unipotent group.  
Let $W$ denote the Weyl group.  Let $\caB$ denote the variety of Borel subalgebras in $\frg$.  There is a $G$ equivariant isomorphism of $\caB$ with $G/B$ (so that $\frb$ corresponds to $eB$).  

For $\alpha \in \Delta$, let $\frg_\alpha$ denote the corresponding root space.  Set
\[
\frb_w = \frh \oplus \frn_w = \frh \oplus \bigoplus_{\alpha \in \Delta^+} \frg_{w\alpha},
\]
and write
\[
\frn_w^\opp = \bigoplus_{\alpha \in \Delta^+} \frg_{-w\alpha}.
\]
for the opposite nilradical.  Finally, let
\[
C_w = B\cdot \frb_w \subset \caB.
\]
Under the identification of $\caB$ with $G/B$, $C_w = B\dot w B$ where $\dot w$ is a representative of $w$ in $N_G(\frh)$.  

The tangent space at $\frb_w$ in $\caB$ canonically identifies with $\frg/\frb_w$.  Since the Lie algebra of the stabilizer in $B$ of the orbit $C_w$ through
$\frb_w$ is $\frb \cap \frb_w$, we have \begin{equation}
\label{e:bruhat-tangent-space}
T_{\frb_w}(C_w) = \frb / (\frb \cap \frb_w)
\end{equation}
canonically.
Thus the image of 
\begin{equation}
\label{e:bruhat-complement}
\frn_w^\opp \cap \frn^\opp
\end{equation}
in $T_{\frb_w}(\caB)$ is a complement to $T_{\frb_w}(C_w)$ in $T_{\frb_w}(\caB)$
and so
\[
S_w = \left \{\exp(A) \cdot \frb_w \; | \; A\in \frn_w^\opp \cap \frn^\opp \right \}
\]
is a smooth subvariety meeting $C_w$ transversally at $\frb_w$, i.e.~a normal slice to $C_w$ at $\frb_w$.  

\begin{remark}
When $G = GL(n,\bbC)$, it's sometimes convenient to use the map $A \mapsto \Id + A$ in place of the exponential map, 
\[
S'_w =  \left \{ (\Id + A) \cdot \frb_w \; | \; A \in \frn_w^\opp \cap \frn^\opp \right \}.
\]
\end{remark}

\begin{remark}
\label{rem:bruhat-slice-coordinates}
For specific applications, it is important to write down the slice in local coordinates.  For example, suppose $G = \GL(n,\bbC)$.  The subspace $\dot w N^\opp \subset G$ is an affine subspace of dimension $\binom{n}{2}$ with an obvious system of coordinates coming from the matrix entries of $N^\opp$.   The restriction of the natural map $G \rightarrow G/B$ to $\dot w N^\opp$ is an open immersion mapping $\dot w$ in $\dot w N^\opp$ to $\dot w B$.  The inverse map $\rho$ taking $\dot w n^\opp B$ to $\dot w n^\opp$ defines a system of coordinates with origin centered at $\dot w B$.  Note that in the identification $\caB$ with $G/B$, 
\[
S'_w = \left \{ (\Id + A) \dot wB\; | \; A \in \frn_w^\opp \cap \frn^\opp \right \} \subset  \dot w N^\opp B.
\]
So its image in $\dot w N^\opp$ under $\rho$ consists of the matrices
\[
 \left \{ (\Id + A) \dot w\; | \; A \in \frn_w^\opp \cap \frn^\opp \right \}.
 \]
 (If we fix $\frb$ to be upper triangular (and take $w=x$), this is the slice denoted $N_x$ after equation (3.1) in \cite[Section 3.2]{williamson}.)

\end{remark}

\subsection{$K$ orbits}
\label{s:korbits}
In the setting of Section \ref{s:bruhat}, let $K$ be the fixed points of an involutive automorphism $\theta$ of $G$. 
Write $\theta$ for the induced involution of $\frg$, and $\frg = \frk \oplus \frp$ for the corresponding eigenspace 
decomposition.  Write $\mu$ for the moment map (for the $G$ action) of $T^*\caB$.

The set of $K$ orbits on $\caB$ is finite.  Write $\Gamma$ for a set parametrizing these orbits.  (This can be made
explicit when necessary.)  For $\gamma \in \Gamma$, let $\frb_\gamma = g_\gamma\cdot \frb = \frh_\gamma \oplus \frn_\gamma$ denote a representative of the $K$ orbit $Q_\gamma$ parametrized by $\gamma.$   We assume, as we may, that $\frb_\gamma$ is chosen so that $\frh_\gamma$ is $\theta$-stable.

Since the Lie algebra of the stabilizer in $K$ of $\frb_\gamma$ is $\frk \cap \frb_\gamma$, the tangent space to $Q_\gamma$ at $\frb_\gamma$  is
\[
T_{\frb_\gamma}(Q_\gamma) = \frk/(\frk \cap \frb_\gamma) \subset \frg/\frb_\gamma = T_{\frb_\gamma}(\caB).
\]
Since $\frh_\gamma$ is $\theta$-stable,
\begin{equation}
\label{e:k-complement}
 \frn^\opp_\gamma\cap \frp
 \end{equation}
 is a complementary subspace to $ \frk/(\frk \cap \frb_\gamma)$ in  $\frg/\frb_\gamma$.  (Note however that $T_{\frb_\gamma}(Q_\gamma)$ need {\em not} identify with $\frn_\gamma^\opp \cap \frk$ in general.)
Thus
\[
S_\gamma = \left \{\exp(A) \cdot \frb_\gamma \; | \; A \in \frn_\gamma^\opp \cap \frp \right \}
\]
is a smooth subvariety meeting $Q_\gamma$ transversally at $\frb_\gamma$, i.e.~a normal slice to $Q_\gamma$ at $\frb_\gamma$.    

\begin{remark}
\label{k-slice-coordinates}
Again when $G = GL(n,\bbC)$, it's sometimes useful to instead use
\begin{equation}
\label{e:upq-slice}
S'_\gamma =  \left \{ (\Id + A) \cdot \frb_\gamma \; | \; A \in \frn_\gamma^\opp \cap \frp \right \}.
\end{equation}
By analogy with the discussion in Remark \ref{rem:bruhat-slice-coordinates}, the slice identifies with the matrices
\begin{equation}
\label{e:upq-slice2}
 \left \{ (\Id + A) g_\gamma \; | \; A \in \frn_\gamma^\opp \cap \frp \right \}
 \end{equation}
 in the affine subspace $g_\gamma N^\opp$ of $G$. 
\end{remark}

\subsection{Setting for applications to cohomological induction} 
\label{s:coh-induction}

In the setting of Section \ref{s:korbits}, let $\frq = \frl\oplus \fru$ denote a $\theta$-stable parabolic of $\frg$ containing the fixed basepoint $\frb \in \caB$, and write $L$ for the normalizer in $G$ of $\frl$.   
(It would have been more customary to use $\mathfrak{q}$
for a $\theta$-stable parabolic, but we have used $Q$ elsewhere.)
Write $\caB^\frq \subset \caB$ for the preimage of $K\cdot \frq$ under the natural projection $\pi$ from $\caB$ to $G\cdot \frq$, the partial flag
variety of $G$ conjugates of $\frq$, taking $\frb$ to $\frq$.

Clearly $\caB^\frq$ is smooth and $K$ invariant, and fibers over $K\cdot \frq$ with fiber equal to the variety $\caB_\frl$ of Borel subalgebras of $\frl$.  More explicitly, note that $\caB_\frl \simeq L/(L \cap B)$ maps injectively into $\caB^\frq$ by mapping $\frb \cap \frl$ to $\frb$ and extending $L$ equivariantly; concretely $\ell (L\cap B) \in L/(L\cap B)$ maps to $\ell B \in G/B$.  Call the resulting map $i_\frq$,
\begin{equation}
\label{e:iq}
i_\frq \; : \; \caB_\frl \lra \caB^\frq.
\end{equation}  
Next fix an $L \cap K$ orbit $O_\frl$ on $\caB_\frl$ through a Borel $\frb_{\frl}$ in $\frl$, and let $O$ denote the (single) $K$ orbit $K\cdot i_\frq(O_\frl)$ through $\frb_\gamma := i_\frq(\frb_\frl)$; so $\frb_{\frl} = \frl \cap \frb_\gamma$.
Write $\frq_\gamma = \frl_\gamma \oplus \fru_\gamma$ for the conjugate of $\frq$ containing $\frb_\gamma$, i.e. $\frq_\gamma = \pi(\frb_\gamma)$.

Since $\pi$ is a fibration, we have that
\[
T_{\frb_\gamma}(\caB^\frq) = T_{\frq_\gamma}(K\cdot \frq) \oplus T_{\frl \cap \frb_\gamma}(\caB_\frl).
\]
We can apply the construction of Section \ref{s:korbits} to find a normal slice $S_\frl$ to $O_\frl$ in $\caB_\frl$ at $\frl \cap \frb_\gamma$.  So we can further decompose the
tangent space as
\[
T_{\frb_\gamma}(\caB^\frq) = T_{\frq_\gamma}(K\cdot \frq_\gamma) \oplus T_{\frl \cap \frb_\gamma}(O_\frl)  \oplus T_{\frl \cap \frb_\gamma}(S_\frl).
\]
We recognize the first two summands as $T_{\frb_\gamma}(O)$.
In other words,
\begin{align*}
T_{\frb_\gamma}(\caB^\frq) 
&= T_{\frb_\gamma}(O) \oplus T_{\frb_\frl}(S_\frl)\\
&= T_{\frb_\gamma}(O) \oplus T_{\frb_\gamma}(i_\frq(S_\frl)).
\end{align*}
It follows that 
\begin{equation}
\label{e:sliceq}
\begin{split}
\text{\em if $S_\frl$ is a normal slice to $O_\frl$ in $\caB_\frl$ at $\frb_\frl$, then}\\
\text{\em $i_\frq(S_\frl)$ is a normal slice to $O$ at $\frb_\gamma = i_\frq(\frb_\frl)$ in $\caB^\frq$}.
\end{split}  
\end{equation}

\subsection{Characteristic cycles and cohomological induction}
\label{s:coh-induction-app}
In the setting of Section \ref{s:coh-induction}, 
define a bijection
\[
i_\frq^* \; : K \bs \caB^\frq \lra (L\cap K) \bs \caB_\frl 
\]
via
\begin{equation}
\label{e:iq-orbits}
i_\frq^*(O) = i_\frq^{-1}\left ( O \cap i_\frq(\caB_\frl) \right ).
\end{equation}
(Since $i_\frq$ is $L \cap K$ equivariant, it is clearly a well-defined injection.  To see $i_\frq^*$ is surjective, given any $L \cap K$ orbit $O_\frl$ in $\caB_\frl$ note
that $O_\frl = i_\frq^*(K\cdot i_\frq(O_\frl))$.)

Let $\caL$ denote an irreducible $K$ equivariant
local system supported on a $K$ orbit $Q$ in $\caB^\frq$ and set $Q_\frl := i_\frq^*(Q)$.
Consider the pullback $i_\frq^*\caL$ of $\caL$ via the restriction of $i_\frq$ to $Q_\frl$.
(It would be more precise to denote this pullback by $(i_\frq |_{Q_\frl})^*\caL$, but we will drop the restriction
from the notation.)
Then $i_\frq^*\caL$ is an irreducible $L \cap K$ equivariant local system supported on $Q_\frl$.  
The characteristic cycles of $i_\frq^*\caL$ and $\caL$ are equivalent in a very simple way:

\begin{prop}
\label{p:coh-induction}
With notation as in the previous paragraph, write
\begin{equation}
\label{e:ccL}
\CC(\IC^\cdot(Q, \caL)) = \sum_{O \in  K \bs \caB} m(O,\caL) [\ol{T_{O}^*\caB}]
\end{equation}
for the characteristic cycle of the intersection cohomology sheaf on the closure $\overline Q$ with coefficients
in $\caL$.  Similarly write
\begin{equation}
\label{e:ccLl}
\CC(\IC^\cdot(Q_\frl, i_\frq^*\caL)) = \sum_{O_\frl \in (L\cap K) \bs \caB_\frl} m_\frl(O_\frl,i_\frq^*\caL) [\ol{T_{O_\frl}^*\caB_\frl}].
\end{equation}
Recall the bijection of orbits given by \eqref{e:iq-orbits}.  Then for each $K$ orbit $O$ on $\caB^\frq$, 
\begin{equation}
\label{e:coh-induction1}
m(O,\caL) = m_\frl(i_\frq^*(O),i_\frq^*\caL).
\end{equation}
Moreover if $O$ is a $K$ orbit on $\caB$ which is not contained in $\caB^\frq$, then
\begin{equation}
\label{e:coh-induction2}
m(O,\caL) = 0.
\end{equation}
\end{prop}
\begin{proof}
Fix a $K$ orbit $O$ in $\caB^\frq$.  Let $S_\frl$ be the normal slice to $O_\frl = i_\frq^*(O)$ constructed in Section \ref{s:coh-induction}.  Then
\[
\overline {Q_\frl} \cap S_\frl \simeq \overline{Q} \cap i_\frq(S_\frl)
\]
with the isomorphism (from left to right) given by $i_\frq$.  Since \eqref{e:sliceq} shows that $i_\frq(S_\frl)$ is a normal slice to $O$, the definition of characteristic cycles in terms of normal
slices \cite[II.6.A]{gm}, then implies \eqref{e:coh-induction1}.
Since the closure of $Q$ is contained in $\caB^\frq$,  \eqref{e:coh-induction2}
 is clear.
\end{proof}

It is useful to translate this into a representation theoretic statement.  
Let $\caD_\frl$ denote the sheaf 
of algebraic differential operators on $\caB_\frl$, and let $Z$ be an irreducible $(\frl,L \cap K)$ module
with trivial infinitesimal character.  
Write
\[
\caZ = \caD_\frl \otimes_{U(\frg)} Z,
\]
a Harish-Chandra sheaf of (holonomic) $(\caD_\frl, L\cap K)$ modules.  Choose an $L \cap K$ invariant
good filtration $\caZ^\bullet$ on $\caZ$.  The associated graded object $gr(\caZ^\bullet)$ defines a coherent sheaf
on $T^*\caB_\frl$ supported on the conormal variety $T^*_{L\cap K}\caB_\frl$.  Its support cycle therefore defines
an $\bbN$-linear combination of conormal bundles to $L \cap K$ orbits on $\caB_\frl$ called the characteristic cycle
of $Z$,
\begin{equation}
\label{e:ccZ}
\CC(Z) = \sum_{O_\frl \in  (L\cap K) \bs \caB_\frl} M_\frl(O_\frl,Z) [\ol{T_{O_\frl}^*\caB_\frl}].
\end{equation}
Let $\caL_{\frl, Z}$ be the irreducible $L \cap K$ equivariant local system corresponding to $Z$ via the Beilinson-Bernstein correspondence.  (The natural correspondence
is with an irreducible $L \cap K$ equivariant flat connection supported on a single $L \cap K$ orbit.  To pass to a local system, here and elsewhere, we will use the bijections explained, for example, in~\cite[Section 2]{ic3}.)
As explained in \cite[II.6.A]{gm},
the terms in \eqref{e:ccLl} and \eqref{e:ccZ} coincide,
\begin{equation}
\label{e:cc-coincideZ}
m_\frl(O_\frl,\caL_{\frl,Z}) = M_\frl(O_\frl,Z).
\end{equation}

Now let $X =\caR_\frq(Z)$ be the (irreducible) $(\frg,K)$ module with trivial infinitesimal character obtained from $Z$ by cohomological induction in the good range 
as in \cite[Section 0.7]{knapp-vogan}, and write
\begin{equation}
\label{e:ccX}
\CC(X) = \sum_{O \in  K \bs \caB} M(O,X) [\ol{T_{O}^*\caB}].
\end{equation}
Once again if $\caL_{X}$ is the irreducible $K$ equivariant local system on $\caB$
 corresponding to $X$, then the terms
 in \eqref{e:ccL} and \eqref{e:ccX} coincide,
\begin{equation}
\label{e:cc-coincideX}
m(O,\caL_X) = M(O,X).
\end{equation}
The key point is that under these hypotheses, $\caL_X$ is supported on $\caB^\frq$ and 
\begin{equation}
\label{e:LP-coh-induction}
\caL_{\frl,Z} = i_\frq^*\caL_X.
\end{equation}
This is well-known to experts.  It follows, for example, by combining the description of cohomological induction in the Langlands classification (e.g. \cite[XI.10]{knapp-vogan}) 
with \cite{ic3}.

Thus \eqref{e:LP-coh-induction} says that we are in the setting of Proposition \ref{p:coh-induction} (with $\caL_X$ playing
the role of $\caL$).  This proposition
together with \eqref{e:cc-coincideZ} and \eqref{e:cc-coincideX} immediately imply

\begin{corollary}
\label{c:coh-induction-rep}
If $Z$ is an irreducible $(\frl,L \cap K)$ module with trivial infinitesimal character and (as above) $X=\caR_\frq(Z)$ is the irreducible $(\frg,K)$ module with trivial
infinitesimal character obtained from $Z$ by  cohomological induction in the good range, then the characteristic cycles of $Z$ and $X$ are related as follows.  In the notation of 
\eqref{e:ccX} and \eqref{e:ccZ},
\begin{equation}
\label{e:coh-induction-rep1}
M(O,X) = M_\frl(i_\frq^*(O),Z)
\end{equation}
for each $K$ orbit $O$ in $\caB^\frq$.
Moreover if $O$ is not contained in $\caB^\frq$, then
\begin{equation}
\label{e:coh-induction-rep2}
 M(O,X) = 0.
\end{equation}
\end{corollary}

\section{main results}
\label{s:main}

Let $G = \GL(n,\bbC)$, fix $p+q=n$, and  let $\theta$ be given by conjugating by the matrix,
\[
\left (
\begin{matrix}
I_p & 0 \\
0 & -I_q
\end{matrix}
\right )
\]
where $I_p$ is the $p \times p$ identity matrix (and similarly for $I_q$).  Then $\frk = \frg\frl(p,\bbC) \oplus \frg\frl(q,\bbC)$ is embedded block diagonally, and $\frp$ consists of the off diagonal blocks.  

In the following theorem, recall that the characteristic variety $\CV(\IC^\cdot(Q,\bbC))$ is defined to be the support of $\CC(\IC^\cdot(Q,\bbC))$.

\begin{theorem}
\label{t:main}
\begin{enumerate}
\item[(a)]
Let $n=8$ and $p=q=4$.  Then there are $K$ orbits $Q_\eta$  and $Q_\gamma$ on $\caB$ of respective dimensions
24 and 20 such that
\[
\CV(\IC^\cdot(Q_\eta,\bbC)) \supset  \overline {T_{Q_\eta}^*\caB} \; \cup \overline {T_{Q_\gamma}^*\caB}.
\]
and
\begin{equation}
\label{e:moment-map-a}
\mu(\ol{T_{Q_\eta}^*\caB}) \supsetneq \mu(\ol{T_{Q_\gamma}^*\caB}).
\end{equation}
In the notation of \cite[Theorem 2.2.8]{yamamoto}, the orbit $Q_\eta$ is parametrized by $(12324341)$, and the orbit $Q_\gamma$ is parametrized by $(12213443)$; see also
Section \ref{s:moment-map}  below.

\item[(b)]
Let $n=12$ and $p=q=6$.  Then there are $K$ orbits $Q$ and $Q'$ on $\caB$ of respective dimensions 42 and 38 such that
\[
\CV(\IC^\cdot(Q,\bbC)) \supset  \overline {T_Q^*\caB} \; \cup \overline {T_{Q'}^*\caB}.
\]
and
\begin{equation}
\label{e:moment-map-b}
\mu(\ol{T_{Q'}^*\caB}) = \mu(\ol{T_Q^*\caB}).
\end{equation}
In the notation of \cite[Theorem 2.2.8]{yamamoto}, $Q$ is parametrized by $1^+2^+(34546563)7^-8^-$, and $Q'$ is parametrized by $1^+2^+(34435665)7^-8^-$; see also
Section \ref{s:moment-map} below.

\end{enumerate}
\end{theorem}
 
\subsection{Proof of Theorem \ref{t:main}(a)}
\label{s:proofa}
We will proceed with a sequence of lemmas.

We begin with a precise description of the orbit $Q_\gamma$ in the theorem.  
Retain the notation of Section \ref{s:coh-induction}.  Let $\frq_\gamma = \frl_\gamma \oplus \fru_\gamma$ denote the $\theta$-stable parabolic defined as the sum of the nonnegative eigenspaces of $\ad(x_\gamma)$ where
\[
x_\gamma \; = \; 
\left (
\begin{matrix}
I_2  \\
& -I_2\\
& & I_2\\
& & & -I_2
\end{matrix}
\right )\; ;
\]
here $\frl_\gamma$ is the zero eigenspace.  

Recall the projection $\pi \; : \; \caB \rightarrow G\cdot \frq$ taking $\frb$ to $\frq$, and let $\frb_\gamma$ denote a generic element in $\pi^{-1}(\frq_\gamma)$.  Set
\begin{equation}
\label{e:orbita}
Q_\gamma = K\cdot \frb_\gamma.
\end{equation}

From the above description of $\frb_\gamma$ and an understanding of $U(1,1)$, we see that we can write $\frb_\gamma = g_\gamma\cdot \frb$ where
\begin{equation}
\label{e:g-gamma}
g_\gamma \: = \; \left (
\begin{matrix}
1 & 0 & 0 & -1 & 0 & 0 & 0 & 0\\
0 & 1 & -1 & 0 & 0 & 0 & 0 & 0\\
0 & 0 & 0 & 0 & 1 & 0 & 0 & -1\\
0 & 0 & 0 & 0 & 0 & 1 & -1 & 0\\
1 & 0 & 0 & 1 & 0 & 0 & 0 & 0\\
0 & 1 & 1 & 0 & 0 & 0 & 0 & 0\\
0 & 0 & 0 & 0 & 0 & 1 & 1& 0\\
0 & 0 & 0 & 0 & 1 & 0 & 0 & 1\\
\end{matrix}
\right )
\end{equation}
Then orbit $Q_\gamma$ is parametrized by what Yamamoto calls the clan $(12213443)$.

Next we turn to understanding a suitable normal slice for $Q_\gamma$ as in Remark \ref{k-slice-coordinates} explicitly in an
affine coordinate patch.  This is summarized as follows.
\begin{lemma}
\label{l:maina}
For $Q_\gamma$ as in \eqref{e:orbita}, the matrices of \eqref{e:upq-slice2} are
\begin{equation}
 \label{e:u44-slice}
\left (
\begin{matrix}
1   &  0&    0&    -1&   0&    0&    0&    0 \\
 0   & 1&    -1&    0&   0&    0&    0&    0\\
 x_2  &  x_1 &   x_1 &   x_2 &   1 &   0  &  0 &   -1\\
 x_4  &  x_3  &  x_3 &   x_4 &   0 &   1  &  -1 &  0 \\
1  &  0 &   0  &  1  &  0 &   0 &   0 &   0\\
0  &  1  &  1 &   0 &   0 &   0  &  0  &  0\\
y_1  &  y_2 &   -y_2 &  -y_1 &  0 &   1 &   1 &   0\\
y_3  &  y_4 &   -y_4 &  -y_3 &  1  &  0 &   0 &   1
\end{matrix}
\right )\; .
\end{equation}
According to  Remark \ref{k-slice-coordinates}, this is the normal slice $S_\gamma'$  in the natural 
coordinates on $g_\gamma N^-$.
\end{lemma}

\begin{proof}[Proof of Lemma \ref{l:maina}]
Consider
\begin{equation}
\label{e:u44-slice-1}
 \{ \Id + A \; | \; A \in \fru^\opp_\gamma \cap \frp \}  \; = \; 
\left (
\begin{matrix}
1 & 0 & 0 &0 & 0 &0 &0 & 0\\
0 & 1 & 0 &0 & 0 &0 &0 & 0\\
0 & 0 & 1 & 0 &x_2 & x_1 & 0 & 0 \\
0 & 0 & 0 & 1 &x_4 & x_3 & 0 & 0 \\
0 & 0 & 0 & 0 &1 & 0 & 0 & 0 \\
0 & 0 & 0& 0 &0& 1& 0 & 0 \\
y_1 & y_2  & 0 &0 & 0 &0 &1 & 0\\
y_3 & y_4& 0 &0 & 0 &0 &0 & 1\\
\end{matrix}
\right )
\end{equation}
where the $x$ and $y$ coordinates are indeterminate complex variables.

Since $\frb_\gamma$ is generic, $\frb^\opp_\gamma \cap \frp = \frq^\opp_\gamma \cap \frp$.  So the set of matrices in \eqref{e:u44-slice-1} acting on $\frb_\gamma$ gives the slice $S'_\gamma$ in $\caB$ given in \eqref{e:upq-slice}.  To realize this slice in the natural coordinates on  $g_\gamma N^\opp$ described in Remark \ref{k-slice-coordinates},
we need to multiply the matrices in \eqref{e:u44-slice-1} on the right by the matrix in \eqref{e:g-gamma}.  This is the 
set of matrices that appears in \eqref{e:u44-slice}.
\end{proof}

\medskip

In the coordinates of Lemma \ref{l:maina}, we aim to see the intersection of the slice $S'_\gamma$ with the closure of larger orbits $Q_\eta$ in order to analyze the singularity of the closure along 
$Q_\gamma$.  For the particular $\eta$ appearing in Theorem \ref{t:main}(a), we summarize the result as follows.

\begin{lemma}
\label{l:maina-2}
Fix $Q_\gamma$ and $Q_\eta$ as in Theorem \ref{t:main}(a).  Then the intersection $Y$ of $\overline Q_\eta$ with the normal
slice $S'_\gamma$ of \eqref{e:upq-slice} in the coordinates of Lemma \ref{l:maina} is cut out by the following equations:
\begin{equation}
\label{e: a}
\mathrm{rank}\left( 
\begin{matrix}
x_1  &  x_2  \\
  x_3  &  x_4\\
  y_2 & y_1\\
  y_4 & y_3 
  \end{matrix}
\right  )\; \leq \; 1
\end{equation}
\begin{equation}
\label{e: b}
\mathrm{rank}\left( 
\begin{matrix}
x_1  &  x_3  \\
  x_2  &  x_4  \\
  y_3  &  y_1  \\
   y_4  &  y_2\\
  \end{matrix}
\right ) \; \leq \; 1.
\end{equation}
\end{lemma}

\begin{remark}
\label{r:alternate-equations}
There is some overlap in the equations given by \eqref{e: a} and \eqref{e: b}.  For example, $x_1x_4-x_2x_3 = 0$ and
$y_1y_4-y_2y_3 = 0$ appear in both.  (The symmetric formulation given in the lemma will be useful below.)  It is easy
to verify the equations given by \eqref{e: a} and \eqref{e: b} are equivalent to the equations given by \eqref{e: a} and
\begin{equation}
\label{e: c}
\det
\left (
\begin{matrix}
x_1 & x_3 \\
y_4 & y_2 
\end{matrix}
\right) \; = \; 
\det \left ( \begin{matrix}
x_1 & x_3 \\
y_3 & y_1 
\end{matrix}
\right) \; = \; 
\det
\left (
\begin{matrix}
x_2 & x_4 \\
y_3 & y_1 
\end{matrix}
\right) \; = \;  0
\end{equation}
This time, there is no overlap between the equations given by \eqref{e: a} and \eqref{e: c}.
\end{remark}

\begin{proof}[Proof of Lemma \ref{l:maina-2}]
Let $Z_\eta$ be the inverse image of $\overline{Q_\eta}$ in $G$ under the natural projection from $G$ to $G/B$.  Just as the case of ordinary Schubert varieties for $\GL(n,\bbC)$, it 
is possible to describe equations on $G$ that cut out $Z_\eta$.  This is done in \cite[Lemma 3.5]{wyser-yong}.  In the notation
of \cite[Section 3]{wyser-yong}, let $Q_\eta$ be parametrized by the clan (12324341).  Following \cite[Section 3.1]{wyser-yong}, there are three kinds of equations to consider.  The first kind (under the heading (ii) in \cite[Section 3.1]{wyser-yong}) amounts to requiring that the upper-left four-by-four submatrix of \eqref{e:u44-slice} has zero determinant,
which reduces to 
\begin{equation}
\label{e:second}
\mathrm{det}\left( 
\begin{matrix}
y_1  &  y_3  \\
  y_2  &  y_4 
  \end{matrix}
\right ) \; = \; 0.
\end{equation}

 The second kind (under the heading (i) in  \cite[Section 3.1]{wyser-yong}) amounts to requiring that the lower-left four-by-four submatrix of \eqref{e:u44-slice} has zero determinant,
which reduces to 
\begin{equation}
\label{e:first}
\mathrm{det}\left( 
\begin{matrix}
x_1  &  x_3  \\
  x_2  &  x_4 
  \end{matrix}
\right  )\; = \; 0.
\end{equation}

The third kind of equations (under the heading (iii) in  \cite[Section 3.1]{wyser-yong}) are indexed by pairs $i<j$, and not all pairs 
give rise to nontrivial equations.   In our example, the ones that do are the pairs $(2,4), (2,5), (2,6), (3,4),(3,5),(3,6),(4,5)$ and $(4,6)$; and for 
each such pair $(i,j)$, one requires that the minors of a certain size (which turns out to be $\mathrm{min}(8,i+j)$ in this example) of a certain 
$8\times(i+j)$ matrix are zero.  When restricted to the slice in \eqref{e:u44-slice}, the $8\times(i+j)$ matrix in question has 
its first $i$ column consisting of the first $i$ columns of the matrix in  \eqref{e:u44-slice} except that the last $4$ rows are 
zeroed out; and the last $j$ columns consists of the first $j$ columns of \eqref{e:u44-slice}.

The lemma can now be verified by computer using Macaulay2: one simply writes down the matrices described the previous paragraph; computes the ideal
generated by the indicated minors; adds the ideal generated by \eqref{e:second} and \eqref{e:first}; and verifies that the result coincides with the ideal generated by the
two-by-two minors of the matrices in \eqref{e: a} and \eqref{e: b}.   (Though it doesn't matter, the two ideals are in fact equal; no radicals need to be taken.)  The details
are given in Section \ref{s:appendix}.  For completeness, we now sketch how to perform these calculations by hand.

The only nontrivial relations (beyond \eqref{e:second} and \eqref{e:first}) corresponding to the pairs $(2,4)$, $(2,5)$, $(2,6)$, $(3,4)$, $(3,5)$ all appear in the conditions 
corresponding to the pair $(2,4)$ in heading (iii) of \cite[Section 3.1]{wyser-yong}.  The requirement is
that the rank $6$ minors of
\begin{equation*}
\left (
\begin{matrix}
1   &  0&   1   &  0&    0&    -1 \\
 0   & 1& 0   & 1&    -1&    0\\  
 x_2  &  x_1 &  x_2  &  x_1 &   x_1 &   x_2 \\ 
 x_4  &  x_3  &    x_4  &  x_3  &  x_3 &   x_4 \\
0 & 0 & 1 &  0 &   0  &1 \\
0 &   0  & 0  &  1  &  1 &0    \\
0  &  0 &    y_1  &  y_2 &   -y_2 &  -y_1 \\
0 &  0 &    y_3  &  y_4 &   -y_4 &  -y_3  
\end{matrix}
\right)
\end{equation*}
\smallskip

\noindent
vanish.  
This is equivalent to \eqref{e:first} and \eqref{e:second} together with 
\begin{equation}
\label{e:third}
\det
\left (
\begin{matrix}
x_1 & x_2 \\
y_2 & y_1
\end{matrix}
\right )
\; = \;
\det
\left (
\begin{matrix}
x_4 & x_3 \\
y_1 & y_2
\end{matrix}
\right )
\; = \;
\det
\left (
\begin{matrix}
x_2 & x_1 \\
y_3 & y_4
\end{matrix}
\right )
\; = \;
\det
\left (
\begin{matrix}
x_4 & x_3 \\
y_3 & y_4
\end{matrix}
\right )
\; = \; 0.
\end{equation}
The six equations given by \eqref{e:second}, \eqref{e:first}, and \eqref{e:third}
are exactly the vanishing of the six two-by-two minors of the matrix in \eqref{e: a}.  Thus \eqref{e:second}, \eqref{e:first}, and \eqref{e:third}
are equivalent to  \eqref{e: a}.

\smallskip

The Wyser-Yong minor equations for the pair $(4,5)$ turn out to be implied for those corresponding to the pair $(3,6)$.  For the latter, the requirement is that
the rank eight minors of the following matrix vanish:
\begin{equation}
\label{e:wy36}
\left (
\begin{matrix}
1   &  0&    0&   1   &  0&    0&    -1&   0& 0\\
 0   & 1&    -1& 0   & 1&    -1&    0&   0& 0\\  
 x_2  &  x_1 &   x_1 &  x_2  &  x_1 &   x_1 &   x_2 &   1& 0\\ 
 x_4  &  x_3  &  x_3 &    x_4  &  x_3  &  x_3 &   x_4&   0& 1\\
0 & 0 & 0 & 1 &  0 &   0  &1 & 0& 0 \\
0 &   0 & 0 & 0  &  1  &  1 &0 &   0& 0   \\
0  &  0 &   0 & y_1  &  y_2 &   -y_2 &  -y_1 &  0& 1\\
0 &  0 &   0 &  y_3  &  y_4 &   -y_4 &  -y_3 &  1& 0 
\end{matrix}
\right ) 
\end{equation}
This is equivalent to \eqref{e:second}, \eqref{e:first}, and \eqref{e:third} together with 
\begin{equation}
\label{e:forth-2} 
\det
\left (
\begin{matrix}
x_1 & x_3 \\
y_4 & y_2 
\end{matrix}
\right) \; = \; 
\det \left ( \begin{matrix}
x_1 & x_3 \\
y_3 & y_1 
\end{matrix}
\right) \; = \; 0.
\end{equation}
The Wyser-Yong equations for the pair $(4,6)$ require the vanishing of the rank eight minors
of the 8-by-10 matrix obtained from \eqref{e:wy36} by adding an additional column consisting 
of the fourth column of \eqref{e:u44-slice} with its last four
entries replaced by zeros.  This turns out to introduce one further equation:
\begin{equation}
\label{e:forth-3}
\det
\left (
\begin{matrix}
x_2 & x_4 \\
y_3 & y_1 
\end{matrix}
\right) \; = \;  0
\end{equation}
Since \eqref{e:forth-2} and \eqref{e:forth-3} are exactly \eqref{e: c}, Remark \ref{r:alternate-equations} completes the proof.
\end{proof}

\vskip .2in

Next we give  an isomorphic realization of $Y$ (which will also be useful when relating $Y$
to the Kashiwara-Saito singularity below).    We need to introduce some notation.
Let $M(2,\bbC)$ denote the set of $2 \times 2$
complex matrices and set $V = M(2,\bbC)^2$.  Write
\begin{equation}
\label{e:jt}
J= \left ( \begin{matrix} 0 & -1 \\ 1 & 0 \end{matrix} \right ) \text{ and }
T = \left ( \begin{matrix} 0 & 1 \\ 1 & 0 \end{matrix} \right ).
\end{equation}

\begin{definition}\label{group}
Define 
\begin{equation*}
H= H_1 \times H_2 \times \bbC^\times :=   \SL(2, \bbC)\times \GL(2,\bbC)\times \bbC^\times.
\end{equation*}

Let $H$  act on pairs of matrices $(A_1,A_2) \in V=M(2,\bbC)^2$ via
\begin{equation}
\label{e:h-action}
(h_1,h_2,z) \cdot (A_1,A_2) =
\left ( z h_2^*A_1h_1^{-1},h_2A_2h_1^{-1} \right );
\end{equation}
here
we define $g^*$ for an invertible matrix with determinant $\Delta$ to be
\[
\left (
\begin{matrix}
a & b \\
c & d
\end{matrix}
\right )^*
\; = \;
\frac{1}{\Delta}
\left (
\begin{matrix}
a & -b \\
-c & d
\end{matrix}
\right ).
\]
Since the action is linear,  $H$ acts on $V^*$ in the usual way: given $\xi \in V^*$, $h\cdot \xi$ evaluated at $v$ is $\xi(h^{-1}\cdot v)$. 

\end{definition}
\vskip .2in

We use the coordinates of the matrix entries to identify $V$ and $V^*$; the pairing of $V^*$ with $V$ then becomes the dot product of matrix entries (that is, the sum
of the products of corresponding matrix entries).  That is, 
\begin{align}\label{e:htransp-action}
V & \to V^*\\ \nonumber
(A,B) &\to \xi_{(A,B)} \text{ where } \xi_{(A,B)}(X, Y) = \text{Tr}(A^{tr} X)  + \text{Tr}(B^{tr} Y).
\end{align}

Under this identification the action of $H$ on $V^*$ becomes
\begin{equation*}
(h_1,h_2,z) \cdot (A_1,A_2) =
\left ( z {({h_2^*}^{-1})}^{tr} A_1 h_1^{tr}, {(h^{-1}_2)}^{tr} A_2 h_1^{tr} \right ).
\end{equation*}

\begin{lemma}\label{new}
Let $\varphi :  \bbC^8 \rightarrow M(2,\bbC)^2$  be given by
$$
\varphi(x_1, \dots, x_4,y_1, \ldots, y_4) \; = \; \left (
\left (
\begin{matrix}
y_3 &-y_4 \\
x_2 & -x_1
\end{matrix}
\right )
\; , \;  
\left (
\begin{matrix}
y_1 &-y_2 \\
-x_4 & x_3
\end{matrix}
\right )
\right ).
$$
\begin{enumerate}
\item The image of $Y$ under the isomorphism $\varphi$ is 
  \begin{align*}
Z  :=  \varphi(Y) =   \{ (A_1,A_2) \in M(2,\bbC)^2& \:   | \; \det(A_1) = \det(A_2) = 0; 
 A_1JA_2^{tr} =  0; \\ & \text{ and }  A_2^{tr} TA_1 = \left (\begin{matrix} 0 & 0\\0 & 0 \end{matrix}\right)  \}.
\end{align*}

\item The group $H$ preserves $Z$ and the smooth locus $Z^\sm $. Moreover,  the $H$ orbit of 
\[
v_\varepsilon \; := \; 
\left(
\left (
\begin{matrix}
0 & \varepsilon \\
0 & 0
\end{matrix}
\right )
 \; , \;
\left (
\begin{matrix}
0 & \varepsilon \\
0 & 0
\end{matrix}
\right ) 
\right ) 
\]
is dense in $Z$ whenever $\varepsilon \neq 0$.
\end{enumerate}
\end{lemma}

\begin{proof}
The statement is elementary, and there are many ways to see this directly.  To verify part (1), it is useful to note that $g^* = T(g^{-1})^{tr}T$ and 
that $\SL(2,\bbC) = \Sp(1,\bbC) = \{ g \; | \; g^{tr}Jg=J\}.$ 
For part (2) one can compute the stabilizer of $v_\varepsilon$ explicitly, for example.
\end{proof}
\vskip .2in

The vector space $\bbC^8$ (with coordinates $x_1,y_1,\dots, x_4,y_4$) inherits an algebraic Whitney stratification (with connected strata) from the $K$ orbits on
$\caB$, and the subvariety $Y$ cut out by the equations of Lemma \ref{l:maina-2} is a union of strata. 
Hence the same is true for its isomorphic realization $Z$. (Note that
since the slice is transverse, $\{0\} \subset \bbC^*$ is itself a stratum.)  So we can 
consider the characteristic cycle of $Z$.  We will use the following criterion due to Braden \cite{braden} to detect its reducibility.

\begin{lemma}
\label{l:braden}

Let $Z \subset V$ be as in  Lemma \ref{new}(1).
If the smooth locus
$X^\sm$ of $X$ and $\{0\}$ intersect microlocally in codimension one, i.e.~that the closures of the conormal bundles (in $V$) to $Z^\sm$ and $\{0\}$ intersect in codimension one.  Then the conormal bundle
to $\{0\}$ appears in the characteristic cycle of $\IC^\cdot(Z,\bbC)$.
\end{lemma}

\begin{proof}[Proof of Lemma \ref{l:braden}]
Once we note that $Z$ is even dimensional and invariant under the action of $\bbC^\times$ (by dilating coordinates), this follows from Corollary 3 of \cite{braden}.
\end{proof}

\smallskip

\noindent We now verify the hypothesis of Lemma \ref{l:braden}.

\smallskip

\begin{lemma}
\label{l:maina-3}
The smooth locus
$Z^\sm$ of $Z$ and $\{0\}$ intersect microlocally in codimension one.
\end{lemma}

\begin{proof}[Proof of Lemma \ref{l:maina-3}]
The action of   $H$  on $V$ and $V^*$ in Definition \ref{group} determines 
an 
action of $H$ on $T^*V = V \times V^*$.  By Lemma \ref{new}
this action preserves  the tangent bundle to $Z$ (or $Z^\sm$), and so also the
conormal bundle $T^*_{Z^\sm}(V)$ and its closure.  Thus $H$ acts on
\[
U := \overline {T_{Z^\sm}^*(V)}\cap T_{\{0\}}^*V \subset T_{\{0\}}^*V = V^*.
\]
We are trying to prove this intersection is codimension 1 in $V^*$.  To do so, we will find an element $\xi$ of $U$ whose  $H$ orbit has dimension $7$.
Since we need to write down the conormal vector $\xi$ explicitly, we need to examine the tangent space to $Z^\sm$.
Recall that $H\cdot v_\varepsilon$ is dense in $Z^\sm$; note
also that the orbit $H\cdot v_\varepsilon$ contains $v_{\varepsilon'}$ for all $\varepsilon' \neq 0$.  

One computes directly from the $H$ action that the tangent space to $H\cdot v_\varepsilon$ at $v_\varepsilon$, viewed as a subspace of $V$, consists of all matrices
\begin{equation}
\label{e:xi-eps}
\left(
\left (
\begin{matrix}
a & b\\
0 & c
\end{matrix}
\right )
 \; , \;
\left (
\begin{matrix}
a & d\\
0 & -c
\end{matrix}
\right ) 
\right ) 
\end{equation}
where $a,b,c, d\in \bbC$.  We now write down a cotangent vector $\xi \in V^*$ that vanishes on the tangent space to $H\cdot v_\varepsilon$ at $v_\varepsilon$ for all 
nonzero $\varepsilon$.
To do so we identify  $V$ and $V^*$ via  \eqref{e:htransp-action}.  With this identification in place, let
\begin{equation}
\label{e:xi-eps2}
\xi \; = \;
\left(
\left (
\begin{matrix}
1 & 0 \\
 1& 1
\end{matrix}
\right )
 \; , \;
\left (
\begin{matrix}
-1 & 0\\
 0 & 1
\end{matrix}
\right ) 
\right ) .
\end{equation}
Then the dot product (using the matrix entries) of \eqref{e:xi-eps} and \eqref{e:xi-eps2} is zero.  Since $H\cdot v_\varepsilon$ is dense in $Z^\sm$ for all nonzero $\varepsilon$,  
$\xi$ is in the 
fiber of the conormal bundle $T^*_{Z^\sm}(V)$ at $v_\varepsilon$ for all nonzero $\varepsilon$.  Passing to the closure, we conclude $\xi \in U$
as desired.  

To finish the proof we need to compute the dimension of $H\cdot \xi$.  We use the identification $V$ and $V^*$ of \eqref{e:htransp-action}  to compute the stabilizer in $H$ of $\xi$,
and find that it is one-dimensional,
\begin{equation}
\mathrm{Stab}_H(\xi) \; = \; \left \{
\pm \left (
\begin{matrix}
1 & \beta \\
 0 &   1
\end{matrix}
\right )
 \; , \;
\pm \left (
\begin{matrix}
1 & -\beta\\
 0 & 1
\end{matrix}
\right ) \; , \; 1
  \right \} \; .
\end{equation}
Since $H$ has dimension 8, the $H$ orbit of $\xi$, and hence also $U$, has dimension 7.
  \end{proof}
\smallskip
 
\noindent The reducibility statement in Theorem \ref{t:main}(a) now follows immediately from the reducibility given by Lemmas \ref{l:braden} and \ref{l:maina-3} using basic properties of characteristic cycles and normal
slices (as in \cite[II.6.A]{gm}).  Equation \eqref{e:moment-map-a} will be proved in Section \ref{s:moment-map} below. \qed

\subsection {Proof of Theorem \ref{t:main}(b)}  
\label{s:proofb}
Let $n=12$ and $p=q=6$.
Let $\frq = \frl \oplus \fru$ denote the $\theta$-stable parabolic defined as the sum of the nonnegative eigenspaces of $\ad(x)$ where
\[
x \; = \; 
\left (
\begin{matrix}
2I_2  \\
&I_8\\
& &-I_2
\end{matrix}
\right )\; ;
\]
here $\frl$ is the zero eigenspace.  
Then $\frl \simeq \frg\frl(2,\bbC) \oplus \frg\frl(8,\bbC) \oplus \frg\frl(2,\bbC)$.  Write these three factors
as $\frl = \frl_1 \oplus \frl_2 \oplus \frl_3$ and correspondingly $\caB_\frl = \caB_{\frl_1} \times \caB_{\frl_2} \times \caB_{\frl_3}$.  Then
\[
L\cap K = \GL(2,\bbC) \times \left ( \GL(4,\bbC) \times \GL(4,\bbC) \right ) \times \GL(2,\bbC);
\]
and write the three indicated factors as 
\[
L\cap K = (L\cap K)_1 \times (L\cap K) _2 \times (L\cap K)_3.
\]
Fix elements $\frb_1$ and $\frb_3$ of $\caB_{\frl_1}$ and  $\caB_{\frl_3}$ (so of course  $(L\cap K)_1\cdot \frb_1=\caB_{\frl_1}$ and similarly for  $(L\cap K)_3\cdot \frb_3$).  Fix representatives $\frb_\eta, \frb_\gamma \in 
\caB_{\frl_2}$ of the $(L \cap K)_2$ orbits $Q_\eta$ and $Q_\gamma$ from Theorem \ref{t:main}(a).  Set
\[
Q = K\cdot (\frb_1, \frb_\eta, \frb_3)
\]
and
\[
Q' = K\cdot (\frb_1, \frb_\gamma, \frb_3).
\]
Then
\[
i_\frq^*(Q) = \caB_{\frl_1} \times Q_\eta \times \caB_{\frl_3}
 \qquad 
i_\frq^*(Q') = \caB_{\frl_1} \times Q_\gamma \times \caB_{\frl_3}.
\]
Since the flag varieties $ \caB_{\frl_i}$ are of course smooth, the reducibility in Theorem \ref{t:main}(b) follows immediately from the reducibility of Theorem \ref{t:main}(a) and 
Proposition \ref{p:coh-induction}.  The description of $Q$ and $Q'$ in terms of $Q_\eta$ and $Q_\gamma$ implies
that the clans corresponding to $Q$ and $Q'$ in Yamamoto's parametrization are related to those for $Q_\eta$ and $Q_\gamma$ in the simple way indicated in the statement of (b).  Equation \eqref{e:moment-map-b} will be proved in Section \ref{s:moment-map}.

\subsection{Moment map images}
\label{s:moment-map}
We collect a few more details about the orbits appearing in Theorem \ref{t:main} and, in particular, compute the relevant 
moment map images.  

In the general setting of Section \ref{s:korbits}, the image of the moment map $\mu$ of the conormal variety $T_K^*\caB$ consists of the nilpotent elements $\caN_\theta$ in $(\frg/\frk)^* \simeq \frp$.   Since $\mu$ is $K$ equivariant and proper, and since there are finitely many orbit of $K$ on $\caN_\theta$, it is easy to see that the image of the closure of a single conormal 
bundle $T^*_Q\caB$ is the closure of a single $K$ orbit on $\caN_\theta$.

Return to the special setting at the beginning of  Section \ref{s:main}.
The clan notation is explained in \cite[Theorem 2.2.14]{yamamoto}.  It's easy to see that the clans that appear in
\cite{yamamoto} are special cases of the parametrization of $K$ orbits in terms of twisted involutions.  In this
setting, the latter are involutions in the symmetric group with certain signed fixed points (e.g.~\cite[Proposition 6.(a)]{trapa}.  Repeated entries in the clan
correspond to indices that are interchanged by the involution.  

For example, the clan $(12324341)$ corresponding
to $Q_\eta$ in Theorem \ref{t:main}(a) corresponds to the involution that interchanges the second and fourth entires, the third and sixth, the fifth and seventh, and the first and eighth,
\begin{equation}
\label{e:sigma-eta}
\sigma_\eta = (24)(36)(57)(18).
\end{equation}
Similarly 
\begin{equation}
\label{e:sigma-gamma}
\sigma_\gamma = (23)(14)(67)(58).
\end{equation}
Yamamoto gives an algorithm to compute the moment map image of the conormal bundle to the orbit parametrized by a given
clan.  A simpler algorithm due to Garfinkle \cite{devra} also computes the image using the twisted involution parametrization
\cite{trapa}.

The orbits of $K$ on $\caN_\theta$ are parametrized by signed Young tableau of signature $(p,q)$ where signs alternate
across rows, modulo interchanging rows of equal length (\cite[Theorem 9.3.3]{cm}).  Using Garfinkle's algorithm and
\cite[Theorem 7.2]{trapa}, we compute that: the dense $K$ orbit in $\mu(\overline{T^*_{Q_\eta}\caB})$
is parametrized by $3_+^1 3_-^1 1_+ 1_-$ (that is, the diagram with 1 row of length 3 beginning with $+$, one row of
3 beginning with $-$, and so on);
and the corresponding (smaller) orbit for $\sigma_\gamma$ is parametrized by $2_+^2 2_-^2$.
This proves the claim about moment map images in Theorem \ref{t:main}(a).

The twisted involutions corresponding to $Q$ and $Q'$ in Theorem \ref{t:main}(b) are
\begin{align*}
\sigma & = (1^+)(2^+)(46)(58)(79)(1\; 10)(11^-)(12^-) \text{ and }\\
\sigma' & = (1^+)(2^+)(45)(36)(89)(7 \; 10)(11^-)(12^-) .
\end{align*}
One computes that this time both moment map images are parametrized by $4_+^2 2_+^2$
proving the assertion about moment map images in Theorem \ref{t:main}(b).

\subsection{Relation with Kashiwara-Saito singularity}
\label{s:ks}
Kashiwara-Saito discovered the first example of a Schubert variety for $\GL(n,\bbC)$ with a reducible characteristic
cycle.  They found a pair of elements of $\sigma_1,\sigma_2 \in S_8$ such that the conormal bundle to the Schubert
variety $Q_2$ parametrized by $\sigma_2$ appears in the characteristic variety of the the Schubert
variety $Q_1$ parametrized by $\sigma_1$.  More details can be found in \cite{braden} and \cite{williamson}.   
The intersection of the closure of $Q_1$ with a normal slice to $Q_2$ can be realized as 
\begin{equation}
\label{e:ks}
Z_\ks := \left \{ (A_0,A_1,A_2,A_3) \in M(2,\bbC)^4 \: | \; \det(A_i) = \det(A_iA_{i+1}) = 0 \quad i \in \bbZ /4 =\{0,1,2,3\}
 \right \};
\end{equation}
see \cite[Section 3.4]{williamson}.
Thus $Z_\ks \subset M(2,\bbC)^4 \simeq  \bbC^{16}$ is eight-dimensional with an interesting singularity at 0.  The realization of 
our four-dimensional singularity $Z \subset M(2,\bbC)^2 \simeq \bbC^8$ from \eqref{e:Z} is related as follows.  Recall $J$ from
\eqref{e:jt}.  The injection
\begin{align*}
M(2,\bbC)^2 &\lra M(2,\bbC)^4 \\
(B_1,B_2) &\mapsto (JB_1^{tr}J,  B_1, JB_2^{tr}J, B_2) 
\end{align*}
takes our $Z$ into the Kashiwara-Saito singularity $Z_\ks$.  

\subsection{Relation with Williamson's examples \cite{williamson}.}
\label{s:williamson}
Let $X$ be the irreducible Harish-Chandra module for $U(6,6)$ with trivial infinitesimal character corresponding to the trivial local system on the orbit 
$Q$ appearing in Theorem \ref{t:main}(b) (e.g.~\cite[Section 2]{ic3}).  Theorem \ref{t:main}(b) and
\eqref{e:cc-coincideX}
show that $\CC(X)$ has (at least) two irreducible components, each lying over the same $K$ orbit in $\caN_\theta$.

This is the setting of \cite[Corollary 4.2]{ltc} which implies that two linearly independent Joseph polynomials appear in the expansion of a certain Goldie rank polynomial.   More
precisely, let $q_{\Ann(X)}$ be the Goldie rank polynomial corresponding to the annihilator of $X$.  Fix a generic $\xi$ in 
the moment map image of the closure of the conormal bundle to $Q$.  The Springer fiber $\mu^{-1}(\xi)$ intersects
$\overline{T_Q^*\caB}$ in a single irreducible component, say $C$.  
Similarly write $C'$ for the component obtained by 
intersecting  $\mu^{-1}(\xi)$ with
$\overline{T_{Q'}^*\caB}$.

To the components $C$ and $C'$ one can attach Joseph polynomials $p_C$ and $p_{C'}$ (originally considered in a different context in
\cite{joseph-av}; see also \cite{joseph-polys}).  By Theorem \cite[Theorem 3.13]{ltc}, the reducibility of \ref{t:main}(b)  is equivalent to 
\begin{equation}
\label{e:joseph}
\text{$p_C$ and $p_{C'}$ both appear in the expression of $q_{\Ann(X)}$ in terms of Joseph polynomials;}
\end{equation}
the moment map condition in \eqref{e:moment-map-b} is crucial in making this conclusion.
According to the main result of \cite{joseph-av} (see also \cite[Theorem 4.1]{ltc}), 
this is equivalent to the simple highest weight module $L(w)$ for $\frg\frl(12,\bbC)$
with trivial infinitesimal character and $\Ann(L(w^{-1})) = \Ann(X)$ having a reducible associated variety\footnote{To distinguish between the microlocal invariant of the characteristic variety and its moment map image, we will call the latter the associated variety.}.
Since the proof of Theorem \ref{t:main}(b) reduces to Theorem \ref{t:main}(a) via Proposition \ref{p:coh-induction},
we have found an example of a simple highest weight module for $\frg\frl(12,\bbC)$ with reducible associated variety by analyzing
a four-dimensional singular variety originating in geometric theory of the real group $U(4,4)$. 

Related to his investigation of $p$-canonical bases of finite Hecke algebras, Williamson recently found a similar example by analyzing an occurrence of the Kashiwara-Saito singularity in Schubert varieties \cite[Section 3.5]{williamson}.  It turns out that our example coincides with Williamson's.  
To see this, \cite{trapa} computes $C$ and $C'$ in terms of the Steinberg parametrization of irreducible components by standard Young tableaux.  One checks that
 tableaux attached to $C$ and $C'$ by  \cite[Theorem 7.2]{trapa} indeed coincide with the two  $Q$-tableaux appearing in \cite[Section 3.5]{williamson}.  This implies that 
 the reducible characteristic cycle that Williamson computes for a Schubert variety in $\GL(12,\bbC)$ is equivalent to \eqref{e:joseph}, and hence equivalent to the reducibility in 
 Theorem \ref{t:main}(b) which (by its proof) is equivalent to the reducibility in Theorem \ref{t:main}(a) via Proposition \ref{p:coh-induction}.    
 In fact, considerations along these lines initially led us
 to investigate the orbits in Theorem \ref{t:main}.

Finally, in \cite[Remark 2.8(3)]{williamson}, Williamson (based on correspondence with Saito) gives an example in $\GL(13,\mathbb{C})$ of a more complicated Schubert variety with 
reducible characteristic cycle as another incarnation of the Kashiwara-Saito singularity.  Following the reasoning above, the reducibility in this example is equivalent to the existence of a Harish-Chandra module for $U(7,6)$  whose characteristic cycle contains two distinct irreducible components with the same moment map image.  It turns out that Proposition \ref{p:coh-induction} applies to show that the reducibility is arising from the closure of a certain $\GL(4,\bbC) \times \GL(5,\bbC)$ orbit on the flag variety for $\GL(9,\bbC)$.  Analyzing the relevant normal slice, once again, leads exactly to the singularity given in Lemma \ref{l:maina-2}. 

\section{appendix: calculations for the proof of Lemma \ref{l:maina-2}}
\label{s:appendix}

In the Macaulay2 output below, we label the matrix described by \cite{wyser-yong} corresponding to a pair $(i,j)$
as {\tt Mij}.   For such a matrix, the ideal of its $\min(8,i+j)$ minors is denoted {\tt Jij}. For example, {\tt M24} corresponds to the pair $(2,4)$,
and the ideal generated by the vanishing of its rank $6$ minors is {\tt J24}.  

The ideal {\tt J1} consists of the Wyser-Yong equations for the pairs $(2,4), (2,5), $ $(2,6), (3,4),$ $(3,5),(3,6),(4,5)$ and $(4,6)$, together
with \eqref{e:first} and \eqref{e:second}.

The matrices {\tt A1} and {\tt A2} are the ones given in \eqref{e: a} and \eqref{e: b}.  The ideal {\tt J2} is generated by the vanishing of the
rank 2 minors of these matrices. 

The output below verifies that {\tt J1} coincides with {\tt J2}.

\begin{verbatim}
Macaulay2, version 1.9

i1 : R=ZZ/101[x_1..x_4,y_1..y_4];

i2 : M24=matrix{{1,0,1,0,0,-1},{0,1,0,1,-1,0},{x_2,x_1,x_2,x_1,x_1,x_2},
{x_4,x_3,x_4,x_3,x_3,x_4},{0,0,1,0,0,1},{0,0,0,1,1,0},
{0,0,y_1,y_2,-y_2,-y_1},{0,0,y_3,y_4,-y_4,-y_3}}



o2 = | 1   0   1   0   0    -1   |
     | 0   1   0   1   -1   0    |
     | x_2 x_1 x_2 x_1 x_1  x_2  |
     | x_4 x_3 x_4 x_3 x_3  x_4  |
     | 0   0   1   0   0    1    |
     | 0   0   0   1   1    0    |
     | 0   0   y_1 y_2 -y_2 -y_1 |
     | 0   0   y_3 y_4 -y_4 -y_3 |

i3 : J24=minors(6,M24);

i4 : M25=matrix{{1,0,1,0,0,-1,0},{0,1,0,1,-1,0,0},
{x_2,x_1,x_2,x_1,x_1,x_2,1},{x_4,x_3,x_4,x_3,x_3,x_4,0},
{0,0,1,0,0,1,0},{0,0,0,1,1,0,0},{0,0,y_1,y_2,-y_2,-y_1,0},
{0,0,y_3,y_4,-y_4,-y_3,1}};

o4  = | 1   0   1   0   0    -1   0 |
      | 0   1   0   1   -1   0    0 |
      | x_2 x_1 x_2 x_1 x_1  x_2  1 |
      | x_4 x_3 x_4 x_3 x_3  x_4  0 |
      | 0   0   1   0   0    1    0 |
      | 0   0   0   1   1    0    0 |
      | 0   0   y_1 y_2 -y_2 -y_1 0 |
      | 0   0   y_3 y_4 -y_4 -y_3 1 |

i5 : J25=minors(7,M25);

i6 : M26=matrix{{1,0,1,0,0,-1,0,0},{0,1,0,1,-1,0,0,0},
{x_2,x_1,x_2,x_1,x_1,x_2,1,0},{x_4,x_3,x_4,x_3,x_3,x_4,0,1},
{0,0,1,0,0,1,0,0},{0,0,0,1,1,0,0,0},{0,0,y_1,y_2,-y_2,-y_1,0,1},
{0,0,y_3,y_4,-y_4,-y_3,1,0}}

o6 = | 1   0   1   0   0    -1   0 0 |
     | 0   1   0   1   -1   0    0 0 |
     | x_2 x_1 x_2 x_1 x_1  x_2  1 0 |
     | x_4 x_3 x_4 x_3 x_3  x_4  0 1 |
     | 0   0   1   0   0    1    0 0 |
     | 0   0   0   1   1    0    0 0 |
     | 0   0   y_1 y_2 -y_2 -y_1 0 1 |
     | 0   0   y_3 y_4 -y_4 -y_3 1 0 |

i7 : J26=minors(8,M26);

i8 : M34=matrix{{1,0,0,1,0,0,-1},{0,1,-1,0,1,-1,0},
{x_2,x_1,x_1,x_2,x_1,x_1,x_2},{x_4,x_3,x_3,x_4,x_3,x_3,x_4},
{0,0,0,1,0,0,1},{0,0,0,0,1,1,0},{0,0,0,y_1,y_2,-y_2,-y_1},
{0,0,0,y_3,y_4,-y_4,-y_3}}

o8 = | 1   0   0   1   0   0    -1   |
     | 0   1   -1  0   1   -1   0    |
     | x_2 x_1 x_1 x_2 x_1 x_1  x_2  |
     | x_4 x_3 x_3 x_4 x_3 x_3  x_4  |
     | 0   0   0   1   0   0    1    |
     | 0   0   0   0   1   1    0    |
     | 0   0   0   y_1 y_2 -y_2 -y_1 |
     | 0   0   0   y_3 y_4 -y_4 -y_3 |


i9 : J34=minors(7,M34);


i10 : M35=matrix{{1,0,0,1,0,0,-1,0},{0,1,-1,0,1,-1,0,0},
{x_2,x_1,x_1,x_2,x_1,x_1,x_2,1},{x_4,x_3,x_3,x_4,x_3,x_3,x_4,0},
{0,0,0,1,0,0,1,0},{0,0,0,0,1,1,0,0},{0,0,0,y_1,y_2,-y_2,-y_1,0},
{0,0,0,y_3,y_4,-y_4,-y_3,1}}

o10 = | 1   0   0   1   0   0    -1   0 |
      | 0   1   -1  0   1   -1   0    0 |
      | x_2 x_1 x_1 x_2 x_1 x_1  x_2  1 |
      | x_4 x_3 x_3 x_4 x_3 x_3  x_4  0 |
      | 0   0   0   1   0   0    1    0 |
      | 0   0   0   0   1   1    0    0 |
      | 0   0   0   y_1 y_2 -y_2 -y_1 0 |
      | 0   0   0   y_3 y_4 -y_4 -y_3 1 |


i11 : J35=minors(8,M35);

i12 : M36=matrix{{1,0,0,1,0,0,-1,0,0},{0,1,-1,0,1,-1,0,0,0},
{x_2,x_1,x_1,x_2,x_1,x_1,x_2,1,0},{x_4,x_3,x_3,x_4,x_3,x_3,x_4,0,1},
{0,0,0,1,0,0,1,0,0},{0,0,0,0,1,1,0,0,0},{0,0,0,y_1,y_2,-y_2,-y_1,0,1},
{0,0,0,y_3,y_4,-y_4,-y_3,1,0}}



o12 = | 1   0   0   1   0   0    -1   0 0 |
      | 0   1   -1  0   1   -1   0    0 0 |
      | x_2 x_1 x_1 x_2 x_1 x_1  x_2  1 0 |
      | x_4 x_3 x_3 x_4 x_3 x_3  x_4  0 1 |
      | 0   0   0   1   0   0    1    0 0 |
      | 0   0   0   0   1   1    0    0 0 |
      | 0   0   0   y_1 y_2 -y_2 -y_1 0 1 |
      | 0   0   0   y_3 y_4 -y_4 -y_3 1 0 |


i13 : J36=minors(8,M36);

i14 : M45=matrix{{1,0,0,-1,1,0,0,-1,0},{0,1,-1,0,0,1,-1,0,0},
{x_2,x_1,x_1,x_2,x_2,x_1,x_1,x_2,1},{x_4,x_3,x_3,x_4,x_4,x_3,x_3,x_4,0},
{0,0,0,0,1,0,0,1,0},{0,0,0,0,0,1,1,0,0},{0,0,0,0,y_1,y_2,-y_2,-y_1,0},
{0,0,0,0,y_3,y_4,-y_4,-y_3,1}}

o14 = | 1   0   0   -1  1   0   0    -1   0 |
      | 0   1   -1  0   0   1   -1   0    0 |
      | x_2 x_1 x_1 x_2 x_2 x_1 x_1  x_2  1 |
      | x_4 x_3 x_3 x_4 x_4 x_3 x_3  x_4  0 |
      | 0   0   0   0   1   0   0    1    0 |
      | 0   0   0   0   0   1   1    0    0 |
      | 0   0   0   0   y_1 y_2 -y_2 -y_1 0 |
      | 0   0   0   0   y_3 y_4 -y_4 -y_3 1 |

i18 : J45=minors(8,M45);

i16 : M46=matrix{{1,0,0,-1,1,0,0,-1,0,0},{0,1,-1,0,0,1,-1,0,0,0},
{x_2,x_1,x_1,x_2,x_2,x_1,x_1,x_2,1,0},{x_4,x_3,x_3,x_4,x_4,x_3,x_3,x_4,0,1},
{0,0,0,0,1,0,0,1,0,0},{0,0,0,0,0,1,1,0,0,0},{0,0,0,0,y_1,y_2,-y_2,-y_1,0,1},
{0,0,0,0,y_3,y_4,-y_4,-y_3,1,0}}

o16 = | 1   0   0   -1  1   0   0    -1   0 0 |
      | 0   1   -1  0   0   1   -1   0    0 0 |
      | x_2 x_1 x_1 x_2 x_2 x_1 x_1  x_2  1 0 |
      | x_4 x_3 x_3 x_4 x_4 x_3 x_3  x_4  0 1 |
      | 0   0   0   0   1   0   0    1    0 0 |
      | 0   0   0   0   0   1   1    0    0 0 |
      | 0   0   0   0   y_1 y_2 -y_2 -y_1 0 1 |
      | 0   0   0   0   y_3 y_4 -y_4 -y_3 1 0 |

i17 : J46=minors(8,M46);


i19 : J1=J24+J25+J26+J34+J35+J36+J45+J46+ideal(x_1*x_4-x_2*x_3, y_1*y_4-y_2*y_3);


i20 : A1=matrix{{x_1,x_2},{x_3,x_4},{y_2,y_1},{y_4,y_3}}

o20 = | x_1 x_2 |
      | x_3 x_4 |
      | y_2 y_1 |
      | y_4 y_3 |

i21 : A2=matrix{{x_1,x_3},{x_2,x_4},{y_3,y_1},{y_4,y_2}}

o21 = | x_1 x_3 |
      | x_2 x_4 |
      | y_3 y_1 |
      | y_4 y_2 |


i22 : J2=minors(2,A1)+minors(2,A2);

i23 : J1==J2

o23 = true

\end{verbatim}


\begin{thebibliography}{Will}

\bibitem[ABV]{abv} J.~D.~Adams, D.~Barbasch, and D.~A.~Vogan, Jr., {\em
The Langlands classification and irreducible characters for real
reductive groups}, Progress in Math, Birkh\"auser (Boston), {\bf 104}(1992).

\bibitem[BB]{bb}{J.~L.~Brylinski and W.~Borho},
Differential Operators on Homogeneous Spaces. I. Irreducibility of the Associated Variety for Annihilators of Induced Modules,
{\em Invent.~Math}, {\bf 69} (1982), 437--476.

\bibitem[Br]{braden} T.~Braden,
On the reducibility of characteristic varieties, {\em Proc.~Amer.~Math.~Soc.}, {\bf 130} (2002), no.~7, 2037--2043.

\bibitem[CMc]{cm}
D.~H.~Collingwood and W.~M.~McGovern, {\em Nilpotent orbits in
semisimple Lie algebras}, Chapman and Hall (London), 1994.

\bibitem[Ga]{devra}
D.~Garfinkle,
The annihilators of irreducible Harish-Chandra modules for 
${\mathrm SU}(p,q)$ and other type $A\sb {n-1}$ groups,
{\em  Amer.~J.~Math.}, {\bf 115} (1993), no.~2, 305--369. 

\bibitem[GM]{gm}
M.~Goresky, R.~MacPherson,
Stratified Morse theory,  Ergebnisse der Mathematik und ihrer Grenzgebiete {\bf 14}(1988), Springer-Verlag (Berlin).

\bibitem[J1]{joseph-av}
A.~Joseph, On the variety of a highest weight module, {\em J.~Algebra}, {\bf 88}(1984), 238--278.

\bibitem[J2]{joseph-polys}
A.~Joseph, 
On the characteristic polynomials of orbital varieties, 
{Ann.~Sci.~�cole~Norm.~Sup.~(4)}, {\bf 22} (1989), no.~4, 569--603.


\bibitem[KS]{ks}
M.~Kashiwara, Y.~Saito, Geometric construction of crystal bases, {\em Duke Math.~J.},
{\bf 89}(1997), no.~1, 9--36.

\bibitem[KL]{kl}
D.~Kazhdan, G.~Lusztig, A topological approach to Springer's representations, 
{\em Adv.~Math.}, {\bf 38}(1980), 222--228.

\bibitem[KnV]{knapp-vogan} A.~W.~Knapp, D.~A.~Vogan, Jr., {\em Cohomological Induction and Unitary Representations}, Princeton Mathematical Series {\bf 45}(1995), 
Princeton University Press (Princeton, NJ).


\bibitem[T1]{trapa} P.~E.~Trapa, Generalized Robinson-Schensted
algorithms for real groups, {\em IMRN}, 1999, no.~15, 803--834.


\bibitem[T2]{ltc} P.~E.~Trapa, Leading-term cycles of Harish-Chandra modules and partial orders on components of the Springer fiber, {\em Compos.~Math.}, {\bf 143} (2007), 
no.~2, 515--540.

\bibitem[Vo]{ic3} D.~A.~Vogan, Jr., Irreducible Characters of semisimple Lie groups III.  Proof of the Kazdhan-Lusztig conjectures in the integral case, {\em Inventiones Math.},
{\bf 71} (1983),no.~ 2, 381--417.

\bibitem[Wi1]{williamson} G.~Williamson, A reducible characteristic variety in Type A, in {\em
Representations of
Reductive Groups.  Proceedings of a conference in honor of David Vogan's 60th birthday}, Progress in Math, Birkh\"auser (Boston), {\bf 231}(2015).

\bibitem[Wi2]
{w:icm} G.~Williamson, Parity sheaves and the Hecke category, {\tt arXiv:1801.00896}

\bibitem[WY]{wyser-yong} B.~J.~Wyser, A.~Yong, Polynomials for $\GL_p \times \GL_q$ orbit closures in the flag variety, Selecta Math. (N.S.), {\bf 20} (2014), no.~4, 1083-1110. 

\bibitem[Ya]{yamamoto} A.~Yamamoto, 
Orbits in the flag variety and images of the moment map
for classical groups  I, {\em  Represent. ~Theory},
{\bf  1} (1997), 329--404.
\end{thebibliography}
\end{document}